\documentclass[12pt]{article}
\usepackage{amsmath} \numberwithin{equation}{section}
\usepackage{amsfonts}
\usepackage{amsthm}

\usepackage{verbatim}
\usepackage{cite}
\newtheorem{lemma}{Lemma}[section]

\newtheorem{definition}{Definition}[section]
\newtheorem{theorem}[definition]{Theorem}
\newtheorem{remark}{Remark}[section]
\begin{document}

\title{Global solutions of  a radiative and reactive  gas with  self-gravitation for
higher-order kinetics}
\author{Xulong Qin
\thanks{\footnotesize{Corresponding author
\newline\indent~~E-mail: ~qin\_xulong@163.com.
\newline\indent~~Tel: +86-20-8411-0126, Fax: +86-20-84037978.
\newline\indent~~2000 \it Mathematics Subject
Classification.} ~35Q30; 35R30; 35R76; 35N10.\newline \indent{\it~
Key Words:}~Radiative and Reactive gases; Free Boundary Problem;
 Self-Gravitation;
 \newline\indent~
 Global existence
} \qquad Zheng-an Yao\\
\it \small Department of Mathematics, Sun Yat-sen University,\\
\it \small Guangzhou 510275, People's Republic of China\\
}
\date{}

\maketitle
%%%%%%%%%%%%%%%%%%%%%%%%%%%%%%%%%%%%%%%%%%%%%%%%%%%%%%%%%%%%%%%%%%%%%%%%%%%%%
\begin{abstract}
The existence of global solutions is established for compressible
Navier-Stokes equations by taking into account the radiative and
reactive processes, when the heat conductivity $\kappa$
($\kappa_1(1+\theta^q)\leq\kappa\leq \kappa_2(1+\theta^q),q\geq 0$),
where $\theta$ is temperature. This improves the previous results by
enlarging the scope of $q$ including the constant heat conductivity.
\end{abstract}
%%%%%%%%%%%%%%%%%%%%%%%%%%%%%%%%%%%%%%%%%%%%%%%%%%%%%%%%%%%%%%%%%%%%%%%%%%%%%%%%%%%%%%%%%%%
\section{Introduction}
\label{intro} In this paper, we study a free boundary problem of
viscous, compressible,  radiative and reactive  gases which are
driven by self-gravitation. The system considered can be described
by the conservation of mass,  the conservation of momentum, the
balance of energy and the  reaction-diffusion
 equations
\begin{equation}\label{eulerequation}
\left\{
\begin{array}{lllllllllllllllllllll}
\rho_t+(\rho u)_y =0,\\[2mm]
(\rho u)_t+(\rho u^2)_y=(-p+\mu u_y)_y+\rho f,\\[2mm]
\rho(e_{t}+ue_{y})=(\kappa\theta_{y})_{y}+(-p+\mu
 u_{y})u_{y}+\lambda \rho \phi z^m,\\[2mm]
 \rho(z_{t}+uz_{y})=(d\rho z_{y})_{y}-\rho \phi z^m,
 ~~~~a(t)<y<b(t),~~t>0,
\end{array}
\right.
\end{equation}
where $\rho=\rho(y,t)$, $u=u(y,t)$, $z=z(y,t)$ denote the density,
velocity,  mass fraction of reactant of flows, respectively. In
addition, the positive constants $\mu$, $d$ and $\lambda$ are the
coefficients of the viscosity, species diffusion and difference in
heat between the reactant and the product. $a(t)$ and $b(t)$ are
free boundaries defined by
\begin{equation}
a'(t)=u(a(t),t), \qquad b'(t)=u(b(t),t).\nonumber
\end{equation}
$f=f(y,t)$ is the external force determined by $f=-U_y$ where $U$
solves the following boundary value problem:
\begin{equation*}
\left\{
\begin{array}{lllllllll}
U_{yy}=G\rho,\quad &a(t)<y<b(t),\\[3mm]
U(d,t)=0,\quad &d=a(t), ~~b(t),
\end{array}
\right.
\end{equation*}
where $G$ is the Newtonian gravitational constant.

Moreover, the reactant rate function $\phi=\phi(\rho,\theta)$ is
given by the Arrhenius-type law
\begin{equation}
\phi(\rho,\theta)=K\rho^{m-1}\theta^{\beta}\exp\bigg(-\frac{A}{\theta}\bigg),\label{rate}
\end{equation}
where $m\geq 1$, $\beta$ is a non-negative constant,  while $K$ and
$A$ are  positive constants.

Finally, the pressure $p$ and the internal energy $e$ can be
formulated as  follows:
\begin{align}
p=R\rho \theta+\frac{a}{3}\theta^4,\qquad
 e=C_v\theta+a\frac{\theta^4}{\rho}.\label{stateequation4}
\end{align}
 where the positive $C_v$ is the
specific heat at constant volume, $a>0$ is the Stefan-Boltzmann
constant and $R$ is the perfect gas one. From the experiments
results for gases at very high temperature, refer to \cite{ZR}, the
heat conductivity is
\begin{equation}
\kappa_1(1+\theta^q)\leq \kappa(\rho,\theta)\leq
\kappa_2(1+\theta^q),\qquad q\geq 0,\label{conductivity1}
\end{equation}
where $\kappa_1$ and $\kappa_2$ are positive constant.

The system \eqref{eulerequation} is to be completed with
initial-boundary conditions. More precisely, we suppose here that
\begin{equation}
(\rho,u,\theta,z)|_{t=0}=(\rho_0(y),u_0(y),\theta_0(y),z_0(y)),\label{initial1}
\end{equation}
and
\begin{equation}
(-p+\mu u_y,k\theta_y,d\rho z_y)(d,t))=(-p_e,0,0), \quad d=a(t),
b(t).\label{boundaryfree1}
\end{equation}
The same initial-boundary problem is proposed by Umehara and Tani in
\cite{umeharaTani,umeharaTani11}. We aim at extending their results
by enlarging the scope of $q$ and $\beta$.

The system \eqref{eulerequation} has been extensively studied from
the wellposedness of global solutions to asymptotic behavior of
solutions under various initial-boundary conditions since the
radiation is an important factor in astrophysics. Among them,
Documet \cite{Ducomet3M99} considered the combined radiative and
reactive case only for $R=0$ and
 then established the existence of  a global-in-time solution
 for the Dirichlet boundary problem   with one order
 kinetics  in \cite{Ducomet99} when $q\geq 4$. In another paper \cite{DucometZlotnikhigher}, Ducomet and
Zlotnik treated the case of higher order
 kinetics for fairly general kinetics law when $q\geq 2$. In addition, Ducomet and Zlotnik \cite{DucometZlotnik}
 investigated  the global-in-time bounds and exponential
 stabilization for solutions in $L^q$ and $H^1$ norms by
 constructing new global Lyapunov functions when $q\geq 1$.
 Moreover, Umehara and
 Tani \cite{umeharaTani} showed the global existence of solutions
 for a  self-gravitating, viscous, radiative and reactive gas
 in the case of the external pressure $p_e>0$
 on the free boundaries
 when $4\leq q\leq 16$ and $0\leq \beta \leq \frac{13}{2}$ and extended that result to $q\geq 3$ and $0\leq
 \beta<
 q+9$ in \cite{umeharaTani11}. Qin-Hu-Wang \cite{qin} filled the gap
 with $2\leq q< 3$ and $0\leq
 \beta<2q+6$. There are also various excellent works on the
 multi-dimensional radiation hydrodynamic equations in this
 direction, see \cite{Donatelli06,Donatelli07,Secchi} and references
 cited therein.

However, those results are far from satisfactory. For example, the
constant heat conductivity, which is reasonable in physics, is
excluded due to the restriction of mathematical techniques.  Our
aim, in this paper, is to show global existence of a radiative and
reactive gas for a free boundary problem including constant
heat-conductivity ($q=0$ in \eqref{conductivity1}). We remark that
similar problems for perfect or real fluid flows have been
extensively studied in
\cite{Bressan,BebernesBressan82,BebernesBressan85,Chen,
ChenHoffTrivisaCPDE,ChenHoffTrivisa1,ChenHoffTrivisa2,ChenWagner,Guo,
kawohl, Jiang,Yanagi98} and references  cited therein by many
researchers. However, those methods adopted in the literature is not
valid for the case of radiative flows, which arise extra
difficulties in mathematical analysis. More precisely, besides the
possible concentration of mass and heat, the radiation equation and
$(\rho,\theta)$ dependence of heat conductivity make it complicated.

More recently, Zhang-Zhang \cite{zhangjianwen} established global
existence result for the system \eqref{eulerequation} to Dirichlet
boundary problem with respect to velocity when $q=0$ and $\beta \in
[0,8]$ . However, it leaves a gap for the free boundary boundary
problem suggested in \cite{umeharaTani, umeharaTani11}. We are
intend to make the gap in this paper. The novelty of this paper is
to provide a new insight into solving radiative and reactive flows.
Following the framework of \cite{umeharaTani} in the context, the
main obstacle of a priori estimate is for $v_x\in
L^{\infty}(0,T;L^2[0,1])$, see Lemma \ref{lemma9999}, which in turn
is coupled with the integrability of temperature, and the
"regularity" of heat conductivity. To overcome the above
difficulties, we introduce the auxiliary function $G(v,\theta)$,
refer to Lemma \ref{lemma666} as in \cite{kawohl,Jiang} and succeed
in getting $\theta \in L^r(0,T; L^{\infty}[0,1]),\quad 0<r<q+5,q\geq
0$. Then, using the standard $L^p$ estimation of parabolic equation,
we further get $\theta \in L^{4n-4}[0,1]$ when $q\geq 0$ and $0\leq
\beta<q+9$. With those essential a priori estimates, we can deduce
 a priori estimates of higher derivatives of $(v,\theta,u,z)$. This methods not only adapt
to the Dirichlet problem in \cite{zhangjianwen} but also may be shed
light on the global existence of various radiative models, which is
our further study.

The outline of the paper is as follows. In section 2, we introduce
the Lagrangian  mass coordinates  to transform the free boundary
problem \eqref{eulerequation},
\eqref{initial1}--\eqref{boundaryfree1} into the fixed boundary
problem\eqref{subequation}--\eqref{initial}, which are equivalent
each other. In section 3, we deduce some a priori estimates for
state parameters $(\rho, u,\theta,z)$,  for which it is necessary to
guarantee the global existence of solutions by extending the local
solutions.
\section{Main Result}

In order to study a free boundary problem conveniently, we introduce
the Lagrangian masss coordinates transformation. Let
\begin{equation*}
x=\int_{a(t)}^y\rho(\xi,t)d\xi, \qquad t=t.\label{transformation}
\end{equation*}
Then, the free boundaries $a(t)$ and $b(t)$ become $x=0$ and
$x=\int_{a(t)}^{b(t)}\rho(\xi,t)d\xi=M$, which denotes the total
mass of fluid, respectively. By the conservation law of mass, we may
assume
$M=\int_{a(t)}^{b(t)}\rho(\xi,t)d\xi=\int_{a(0)}^{b(0)}\rho(\xi,0)d\xi=1$,
i.e., $0\leq x\leq 1$. So  the system \eqref{eulerequation} becomes
in Lagrangian mass coordinates

\begin{subequations}\label{subequation}
\begin{align}
&v_t=u_x,\label{sub1}\\[3mm]
&u_t=\left(-p+\frac{\mu
u_x}{v}\right)_x-G\left(x-\frac12\right),\label{sub2}\\[3mm]
&e_t=\left(\frac{\kappa}{v}\theta_x\right)_x+\left(-p+\frac{\mu
u_x}{v}\right)u_x+\lambda \phi
z^m,\label{sub3}\\[3mm]
&z_t=\left(\frac{d}{v^2}z_x\right)_x-\phi
z^m,~~~0<x<1,~~t>0.\label{sub4}
\end{align}
\end{subequations}
The corresponding boundary  conditions are
\begin{equation}
\left(-p+\frac{\mu
u_x}{v},\frac{\kappa}{v}\theta_x,\frac{d}{v^2}z_x\right)\bigg|_{x=0,1}=(-p_e,0,0),
\label{boundary1}
\end{equation}
 and the initial conditions are
\begin{equation}
(v,u,\theta, z)|_{t=0}=(v_0(x),u_0(x),\theta_0(x),z_0(x)),~~~x\in
[0,1]. \label{initial}
\end{equation}
where $v=1/\rho$ is the specific  volume and $u-\int_0^1udx$ is
instead of $u$. We remark that the external  force $f$ is more
realistic than that in  \cite{Ducomet02} by Ducomet.

The main result is stated as follows.
%%%%%%%%%%%%%%%%%%%%%%%%%%%%%%%%%%%%%%%%%%%%%%%%%%%%%%%%%%%%%%%%%%%%%%%%%%%%%%%%%%%%%%
\begin{theorem}\label{thm}
 Let \eqref{rate}--\eqref{conductivity1} hold,
 and suppose that   $(q,\beta)\in [0,+\infty)\times [0,q+9), p_e>0$.
Moreover, the heat-conductivity $\kappa$ in \eqref{conductivity1}
satisfies
\begin{equation*}
|\kappa_{\rho}(\rho,\theta)|\leq \kappa_2(1+\theta)^q,\qquad
|\kappa_{\theta}(\rho,\theta)|\leq \kappa_2(1+\theta)^{q-1}.
\end{equation*}
 Assume that
\begin{equation*}
(v_0(x), u_0(x), \theta_0(x),z_0(x))\in
C^{1+\alpha}([0,1])\times(C^{2+\alpha}([0,1]))^3,
\end{equation*}
and
\begin{equation*}
v_0(x),~~\theta_0(x)>0,~~~~~0\leq z_0(x)\leq 1,~~~~~x\in [0,1],
\end{equation*}
where $\alpha \in (0,1)$. Then there exists a unique solution $(v,
u,\theta,z)$ of the initial boundary problem
\eqref{subequation}--\eqref{initial} for any fixed $T>0$, such that
\begin{equation*}
(v,v_x, v_t)\in (C^{\alpha,\alpha/2}(Q_{T}))^3, \qquad
(u,\theta,z)\in (C^{2+\alpha,1+\alpha/2}(Q_T))^3, \label{thmholder}
\end{equation*}
and
\begin{equation*}
v(x,t),\quad \theta(x,t)>0, \qquad 0\leq z(x,t)\leq 1, \qquad x\in
\overline{Q}_T,
\end{equation*}
 where ${Q}_T=(0,1)\times (0,T)$.

Here, as usual, $C^{m+\alpha}[0,1]$  denotes the H\"{o}lder space on
$[0,1]$ with exponents $m \in N^{+},\alpha\in (0,1)$, $C^{\alpha,
\alpha/2}(Q_T)$ the H\"{o}lder space  on $Q_T=[0,1]\times[0,T]$ with
exponent $\alpha$ in $x$ and
 $\alpha/2$ in $t$ with $\alpha\in (0,1)$. Similarly, if a function $f(x,t)\in
 C^{2+\alpha,1+\alpha/2}$ with $\alpha \in (0,1)$, then it
 means $(u_{xx},u_{t})\in C^{\alpha, \alpha/2}(Q_T)$.
\end{theorem}
%%%%%%%%%%%%%%%%%%%%%%%%%%%%%%%%%%%%%%%%%%%%%%%%%%%%%%%%%%%%%%%%%%%%%%%%%%%%%%%%%%%%%%%%%%%%%%%%%%%
\begin{remark}  The heat conductivity in \eqref{conductivity1} covers
the case of \cite{qin,umeharaTani,umeharaTani11} by
$\kappa=\kappa_1+\kappa_2\frac{\theta^q}{\rho}$.
\end{remark}
\begin{remark}\label{rem22}
The external force $p_e>0$ is not essential. The result are also
satisfactory for $p_e\in R$ provided that we get $v\in L^1[0,1]$,see
Appendix below.
\end{remark}
%%%%%%%%%%%%%%%%%%%%%%%%%%%%%%%%%%%%%%%%%%%%%%%%%%%%%%%%%%%%%%%%%%%%%%%%%%%%%%%%%%%%%%%%%%%%%%%%%%%
The existence and unique proof of Theorem \ref{thm} is motivated by
the procedure devised in \cite{DafermosHisao} through the
application of the Leray-Schauder fixed point theorem, see also
\cite{DafermosHisao,kawohl}. Therefore, a priori estimates of the
local solution  $(\rho, u,\theta,z)$ and its derivatives are
necessary.
%%%%%%%%%%%%%%%%%%%%%%%%%%%%%%%%%%%%%%%%%%%%%%%%%%%%%%%%%%%%%%%%%%%%%%%%%%%%%%%%%%%%%%%%%%%%%%%%%
\section{Proof of Theorem \ref{thm}}
In this section, we establish  some a priori  estimates for the
state parameters $(\rho, u,\theta,z)$ so as to get the global
existence of solutions to the problem
\eqref{subequation}--\eqref{initial}. In the sequel, the generic
positive constant $C(C(T))$ may be different from line to line. In
addition, we denote $|u|^{(0)}:=\sup_{(x,t)\in
\overline{Q}_T}|u(x,t)|$ and $\beta^{+}=\max\{\beta,0\}$.

The following a priori estimates are basic and useful. Please refer
to \cite{umeharaTani} for detailed proofs.
\begin{lemma}Under the hypotheses of Theorem \ref{thm}, it satisfies
that
\begin{equation}
\int_0^1\left(\frac{u^2}{2}+e+\lambda
z+\frac{Gx(1-x)v}{2}\right)dx+p_e\int_0^1vdx\leq C,\label{lemma111}
\end{equation}
\begin{equation*}
U(t)+\int_0^tV(s)ds\leq C(T), \label{Lemma31}
\end{equation*}
and
\begin{equation*}
\begin{split}
&0\leq z(x,t)\leq 1,\\
&\int_0^1\frac{z^2}{2}dx+\int_0^t\int_0^1\frac{d}{v^2}z_x^2dxds+\int_0^t\int_0^1\phi
z^{m+1}dxds=\int_0^1\frac{z^2_0(x)}{2}dx, \label{Lemma332}
\end{split}
\end{equation*}

where
\begin{equation}
U(t)=\int_0^1\left[C_v(\theta-1-\log \theta)+R(v-1-\log
v)\right]dx,\notag
\end{equation}
and
\begin{equation}
V(t)=\int_0^1\left(\frac{\mu
u_x^2}{v\theta}+\frac{\kappa\theta_x^2}{v\theta^2}+\frac{\lambda
\phi z^m}{\theta}\right)dx,\notag
\end{equation}
for $0\leq t\leq T$.
\end{lemma}

%%%%%%%%%%%%%%%%%%%%%%%%%%%%%%%%%%%%%%%%%%%%%%%%%%%%%%%%%%%%%%%%%%%%%%%%%%%%%%%%%%%%%%%%%%%%%%%%%%%%%%%%%%%%
%%%%%%%%%%%%%%%%%%%%%%%%%%%%%%%%%%%%%%%%%%%%%%%%%%%%%%%%%%%%%%%%%%%%%%%%%%%%%%%%%%%%%%%%%%%%%%%%%%%%%%%%%%%%
In the sequence, it gives the upper bound of density.
\begin{lemma}\label{lemma5}
 There holds that
\begin{equation}
C_1(T)\leq v(x,t)\Longleftrightarrow \rho(x,t)\leq \frac{1}{C_1(T)},
\label{lemma51}
\end{equation}
for $0<x<1$ and $0<t<T$.
\end{lemma}
\begin{proof}
Integrating \eqref{sub2} over $[0,x]\times [0,t]$, we obtain
\begin{equation}
\frac{1}{\mu}\int_0^x(u-u_0)d\xi=\log v-\log
v_0-\frac{1}{\mu}\int_0^tpd\tau+\frac{Gx(1-x)t}{2\mu}+\frac{p_e}{\mu}t,\notag
\end{equation}
that is
\begin{equation}
\log v =\log
v_0+\frac{1}{\mu}\int_0^tpd\tau-\frac{Gx(1-x)t}{2\mu}+\frac{1}{\mu}\int_0^x(u-u_0)d\xi-\frac{p_e}{\mu}t.
\label{lemma441}
\end{equation}
Thus, we find by \eqref{lemma111}  that
\begin{equation}
\begin{split}
\log v &\geq \log
v_0+\frac{1}{\mu}\int_0^x(u-u_0)d\xi-\frac{Gx(1-x)t}{2\mu}-\frac{p_e}{\mu}t\\
&\geq \log v_0-C\int_0^1u^2dx-\frac{1}{\mu}\bigg(\frac{G}{8}+p_e\bigg)t-C\\
&\geq \log v_0-Ct-C,\notag
\end{split}
\end{equation}
which leads to \eqref{lemma51}.
\end{proof}
%%%%%%%%%%%%%%%%%%%%%%%%%%%%%%%%%%%%%%%%%%%%%%%%%%%%%%%%%%%%%%%%%%%%%%%%%%%%%%%%

%%%%%%%%%%%%%%%%%%%%%%%%%%%%%%%%%%%%%%%%%%%%%%%%%%%%%%%%%%%%%%%%%%%%%%%%%%%%%%%%%%%%%%%%%%%%%%%%%%%
The following auxiliary result  stems from the  idea of
\cite{umeharaTani}.
%%%%%%%%%%%%%%%%%%%%%%%%%%%%%%%%%%%%%%%%%%%%%%%%%%%%%%%%
\begin{lemma} \label{Lemma52}One has
\begin{equation*}
\int_0^t\max_{[0,1]}\theta^r(x,s)ds\leq C(T),
\end{equation*}
where $0\leq r\leq q+4$ and $q\geq 0$.
\end{lemma}
\begin{proof}
It is known from the mean value theorem
\begin{equation*}
\theta(x(t),t)=\int_0^1\theta dx,
\end{equation*}
 where $x(t)\in [0,1]$ for every $t\in [0,T]$. Furthermore, we
arrive at
\begin{equation}
\begin{split}
\theta(x,t)^{\frac{r}{2}} &=\left(\int_0^1\theta
dx\right)^{\frac{r}{2}}
+\frac{r}{2}\int_{x(t)}^x\theta(\xi,t)^{\frac{r}{2}-1}\theta_{\xi}(\xi,t)d\xi\\[2mm]
&\leq
C\left(1+\int_0^1\frac{\kappa^{\frac12}|\theta_x|}{v^{\frac12}\theta}
\cdot\frac{v^{\frac12}\theta^{\frac{r}{2}}}{\kappa^{\frac12}}dx\right)\\
&\leq
C\left[1+\left(\int_0^1\frac{v\theta^r}{\kappa}dx\right)^{\frac12}V(t)^{\frac12}\right]\\
&\leq
C\left[1+\left(\int_0^1(v+v\theta^4)dx\right)^{\frac12}V(t)^{\frac12}\right],
\label{lemmar4}
\end{split}
\end{equation}
due to $0\leq r\leq q+4$ and Young inequality. Thus it is complete
by integrating \eqref{lemmar4} with respect to time.
\end{proof}
%%%%%%%%%%%%%%%%%%%%%%%%%%%%%%%%%%%%%%%%%%%%%%%%%%%%%%%%%%%%%%%%%%%%%%%%%%%%%%%%
%%%%%%%%%%%%%%%%%%%%%%%%%%%%%%%%%%%%%%%%%%%%%%%%%%%%%%%%%%%%%%%%%%%%%%%%%%%%%%%%%
Similarly, we obtain the lower bound of density, which is  very
important to obtain the  existence of global solutions.
\begin{lemma}
If $q\geq 0$, then it holds that
\begin{equation*}
v(x,t)\leq C(T)\Longleftrightarrow \rho(x,t)\geq 1/C(T),
\end{equation*}
for $(x,t)\in (0,1)\times (0,T)$.
\end{lemma}
\begin{proof}
It follows from \eqref{lemma441} that
\begin{equation}
\begin{split}
\log v &=\log
v_0+\frac{1}{\mu}\int_0^tpd\tau-\frac{Gx(1-x)t}{2\mu}+\frac{1}{\mu}\int_0^x(u-u_0)d\xi-\frac{p_e}{\mu}t\\
&\leq \log
v_0+C\int_0^t\max_{[0,1]}(\theta+\theta^4)ds\\
&~~~+C\int_0^1u^2dx+C\int_0^1u_0^2dx+C(T)\\
&\leq \log v_0+C(T), \notag
\end{split}
\end{equation}
by   Lemmas \ref{lemma5} and \ref{Lemma52}. Thus it is complete.
\end{proof}
%%%%%%%%%%%%%%%%%%%%%%%%%%%%%%%%%%%%%%%%%%%%%%%%%%%%%%%%%%%%%%%%%%%%%%%%%%%%%%%%%%
%%%%%%%%%%%%%%%%%%%%%%%%%%%%%%%%%%%%%%%%%%%%%%%%%%%%%%%%%%%%%%%%%%%%%%%%%%%%%%%%%
\begin{lemma}
Similarly, we have when $q\geq0$
\begin{equation}
\int_0^t\int_0^1u_x^2dxds\leq C(T).\label{Lemma71}
\end{equation}
\end{lemma}
\begin{proof}
The equality \eqref{sub2} implies by  integrating from
$(0,1)\times(0,t)$ that
\begin{equation*}
\begin{split}
&\int_0^1\left(\frac{u^2}{2}+\frac{Gx(1-x)v}{2}\right)dx
+\mu \int_0^t\int_0^1\frac{ u_x^2}{v}dxds\\
 =&\int_0^1\left(\frac{u^2}{2}+\frac{Gx(1-x)v}{2}\right)(x,0)dx+p_e\int_0^1v_0(x)dx\\
 &+\int_0^t\int_0^1pu_xdxds-p_e\int_0^t\int_0^1vdxds\\
 \leq &C(T)+\varepsilon
 \int_0^t\int_0^1u_x^2dxds+\int_0^t\int_0^1p^2dxds,
\end{split}
\end{equation*}
which implies for fixed  small $\varepsilon$ that
\begin{equation*}
\begin{split}
\int_0^t\int_0^1u_x^2dxds &\leq
C(T)\left(1+\int_0^t\max_{[0,1]}\theta^4\left(\int_0^1\theta^4dx\right)ds\right)\\
&\leq C(T).
\end{split}
\end{equation*}
 The proof is complete.
\end{proof}
%%%%%%%%%%%%%%%%%%%%%%%%%%%%%%%%%%%%%%%%%%%%%%%%%%%%%%%%%%%%%%%%%%%%%%%%%%%%%%%%%%%%%%%%%
%%%%%%%%%%%%%%%%%%%%%%%%%%%%%%%%%%%%%%%%%%%%%%%%%%%%%%%%%%%%%%%%%%%%%%%%%%%%%%%%%%%%%%%%%%%%%
\begin{lemma}\label{lemma666}
For $0<\varepsilon<1$ and $q\geq 0$, there holds
\begin{equation}
\int_0^t\int_0^1\frac{\kappa\theta_x^2}{\theta^{1+\varepsilon}}dxds\leq
C.
\end{equation}
\end{lemma}
\begin{proof}
Let $G(v,\theta)=\int_0^{\theta}\xi^{-\varepsilon}
e_{\xi}(v,\xi)d\xi$. Then
\begin{equation*}
\begin{split}
&G_t+(1-\varepsilon)\int_0^{\theta}\xi^{-\varepsilon}p_{\xi}d\xi
\cdot u_{x}=\frac{\mu
u_x^2}{v\theta^{\varepsilon}}+\left(\frac{\kappa}{v}\theta_x\right)_x\cdot\theta^{-\varepsilon}+\lambda
\phi z^m\theta^{-\varepsilon}
\end{split}
\end{equation*}
Integrating it over $[0,1]\times[0,t]$, we arrive at
\begin{equation*}
\begin{split}
&\varepsilon\int_0^t\int_0^1\frac{\kappa\theta_x^2}{v\theta^{1+\varepsilon}}dxds+
\int_0^t\int_0^1\left(\frac{\mu
u_x^2}{v\theta^{\varepsilon}}+\frac{\lambda \phi
z^m}{\theta^{\varepsilon}}\right)dxds\\
&=\int_0^1Gdx-\int_0^1G_0dx+(1-\varepsilon)\int_0^t\int_0^1\int_0^{\theta}\xi^{-\varepsilon}p_{\xi}d\xi
\cdot u_{x}dxds\\
&\leq
C+C\int_0^t\int_0^1\theta^{8}dxds+C\int_0^t\int_0^1u_x^2dxds\\
&\leq C.
\end{split}
\end{equation*}
\end{proof}
%%%%%%%%%%%%%%%%%%%%%%%%%%%%%%%%%%%%%%%%%%%%%%%%%%%%%%%%%%%%%%%%%%%%%%%%%%%%%%%%%%%%%%%%%%%%%%
With the proceeding  a priori estimates at hand, we can firstly
improve the integrability of temperature and deduce  that
\begin{equation}
\theta \in L^r(0,T; L^{\infty}[0,1]),\quad 0<r<q+5,\,\,q\geq
0,\label{theta5}
\end{equation}
when the function
$\widetilde{V}(t)=\int_0^1\frac{\kappa\theta_x^2}{\theta^{1+\varepsilon}}dx$
is instead of $V(t)$ in the proof of Lemma \ref{Lemma52}.
%%%%%%%%%%%%%%%%%%%%%%%%%%%%%%%%%%%%%%%%%%%%%%%%%%%%%%%%%%%%%%%%%%%%%%%%%%%%%%%%%%%%%%%%%%%%%%%%%%%%%%%
Furthermore, we can make  a priori estimates of higher integrability
of $\theta$.
%%%%%%%%%%%%%%%%%%%%%%%%%%%%%%%%%%%%%%%%%%%%%%%%%%%%%%%%%%%%%%%%%%%%%%
%%%%%%%%%%%%%%%%%%%%%%%%%%%%%%%%%%%%%%%%%%%%%%%%%%%%%%%%%%%%%%%%%%%%%%%%
\begin{lemma}
If $0\leq \beta< q+9$ and $q\geq 0$, then it has
\begin{equation}
\begin{split}
\int_0^1\theta^{4n-4}dx+\int_0^t\max_{[0,1]}\theta^{q+4n-4}ds
+\int_0^t\int_0^1\left(|u_x|^n+\kappa\theta^{4n-9}\theta_x^2\right)dxds
\leq C_n,\label{lemma11111}
\end{split}
\end{equation}
for sufficiently large $n>2$.
\end{lemma}
\begin{proof}
Multiplying $\theta^{4n-8}$ on \eqref{sub3}, we get
\begin{equation}
\begin{split}
&\int_0^1\theta^{4n-4}dx+\int_0^t\int_0^1\frac{\kappa\theta^{4n-9}\theta_x^2}{v}dxds\\
&\leq
C+C\int_0^t\int_0^1\left(\theta^{4n-8}u_x^2+\theta^{4n-4}\left|u_x\right|+\theta^{\beta+4n-8}\right)dxds\\
&\leq C+C_n\int_0^t\int_0^1\left|u_x\right|^ndxds
+C_n\int_0^t\max_{[0,1]}(\theta^4+\theta^{(\beta-4)_{+}})\int_0^1\theta^{4n-4}dxds.\label{lemma1111}\\
\end{split}
\end{equation}
On the other hand, we can rewrite \eqref{sub2} with the initial
boundary data
\begin{equation}
\left\{
\begin{array}{llllllllll}
w_t=\frac{\mu}{v}w_{xx}-p+\frac{Gx(1-x)}{2}+p_e,\\[2mm]
w |_{t=0}=w_0(x),\\[2mm]
w |_{x=0,1}=0,\notag
\end{array}
\right.
\end{equation}
where $w=\int_0^xu d\xi$. The standard $L^p$ estimates of solutions
to linear parabolic  problem give rise to
\begin{equation*}
\begin{split}
\int_0^t\int_0^1|u_x|^ndxds&=\int_0^t\int_0^1|w_{xx}|^ndxds\\
&\leq C \left(1+\int_0^t\int_0^1 p^ndxds\right)\\
&\leq C \left(1+\int_0^t\int_0^1 \theta^{4n}dxds\right)\\
 &\leq C \left(1+\int_0^t\max_{[0,1]}\theta^4\int_0^1
 \theta^{4n-4}dxds\right).\label{u}
\end{split}
\end{equation*}
which together with \eqref{lemma1111} leads to
\begin{equation*}
\begin{split}
&\int_0^1\theta^{4n-4}dx+\int_0^t\int_0^1\kappa\theta^{4n-9}\theta_x^2dxds\\
&\leq
C+C\int_0^t\max_{[0,1]}(\theta^4+\theta^{(\beta-4)_{+}})\int_0^1\theta^{4n-4}dxds,
\end{split}
\end{equation*}
which implies $\theta\in L^{4n-4}[0,1]$ by Gr\"{o}nwall's inequality
which in turn, it gives $u_x\in L^{n}\left((0,1)\times (0,t)\right)$
according to \eqref{u} and in the same manner, it implies that
$\theta\in L^{q+4n-4}(0,T;L^{\infty}[0,1])$. Thus, it finishes the
proof.
\end{proof}
%%%%%%%%%%%%%%%%%%%%%%%%%%%%%%%%%%%%%%%%%%%%%%%%%%%%%%%%%%%%%%%%%%%%%%%%%%%%%%%%%%%%%%

%%%%%%%%%%%%%%%%%%%%%%%%%%%%%%%%%%%%%%%%%%%%%%%%%%%%%%%%%%%%%%%%%%%%%%%%%%%%%%%%%%%%

\begin{lemma}\label{lemma9999}
If $0\leq \beta< q+9$ and $q\geq 0$, we arrive at
\begin{equation*}
\int_0^1v_x^2dx+\int_0^t\int_0^1\theta v_x^2dxds\leq
C,\label{Lemma72}
\end{equation*}
\end{lemma}
%%%%%%%%%%%%%%%%%%%%%%%%%%%%%%%%%%%%%%%%%%%%%%%%%%%%%%%%%%%%%%%%%%%%%%%%%%%%%%%%%%%%%%%
\begin{proof}
Multiplying \eqref{sub2} by $u-\frac{\mu}{v}v_x$ and integrating it
over $(0,1)$, it yields
\begin{equation}
\begin{split}
&\frac12\frac{d}{dt}\int_0^1\left(u-\frac{\mu}{v}v_x\bigg)^2dx+\int_0^1
\bigg(\frac{\mu R\theta}{v^3}\right)v_x^2dx\\[2mm]
=&\int_0^1\frac{R\theta uv_x}{v^2}dx
-\int_0^1\left[\left(\frac{R}{v}
+\frac{4a}{3}\theta^3\right)\theta_x
+G\left(x-\frac12\right)\right]\left(u-\frac{\mu}{v}v_x\right)dx.\label{lemma8theta}
\end{split}
\end{equation}
 On one hand, we deduce that
\begin{equation}
\begin{split}
&\int_0^1\frac{R\theta uv_x}{v^2}dx\\[2mm]
&\leq \varepsilon \int_0^1\theta
v_x^2dx+C_{\varepsilon}\max_{[0,1]}\theta\cdot\int_0^1u^2dx\\[2mm]
&\leq \varepsilon \int_0^1\theta
v_x^2dx+C_{\varepsilon}\max_{[0,1]}\theta(x,t).\label{lemma8v}
\end{split}
\end{equation}
On the other hand, we get
\begin{equation*}
\begin{split}
&\left|\int_0^1\left[\left(\frac{R}{v}+\frac{4a}{3}\theta^3\right)\theta_x
+G\left(x-\frac12\right)\right]\left(u-\frac{\mu}{v}v_x\right)dx\right|\\[2mm]
&\leq  C
\left[1+\int_0^1\frac{\theta_x^2}{\theta^2}dx+\int_0^1\theta^6\theta_x^2dx
+\int_0^1(1+\theta^2)\left(u-\frac{\mu}{v}v_x\right)^2dx\right]\\[2mm]
&\leq  C
\left[1+\int_0^1\frac{\kappa\theta_x^2}{v\theta^2}dx
+\int_0^1\kappa\theta^{4n-9}\theta_x^2dx+\int_0^1(1+\theta^2)(u^2+v_x^2)dx\right]\\[2mm]
&\leq  C
\left[1+\int_0^1\frac{\kappa\theta_x^2}{v\theta^2}dx+\int_0^1\kappa\theta^{4n-9}\theta_x^2dx
+\left(1+\max_{[0,1]}\theta^2\right)\cdot\left(1+\int_0^1v_x^2dx\right)\right],
\end{split}
\end{equation*}
which together with \eqref{lemma8theta} and \eqref{lemma8v} for
any fixed small $\varepsilon$ leads to
\begin{equation*}
\begin{split}
&\int_0^1v_x^2dx+\int_0^t\int_0^1\theta v_x^2dxds\\[2mm]
 &\leq
C+C\int_0^t\left(\int_0^1\frac{\kappa\theta_x^2}{v\theta^2}dx
+\int_0^1\kappa\theta^{4n-9}\theta_x^2dx+\max_{[0,1]}\theta^2(x,s)\right)ds\\[2mm]
&+
C\int_0^t\left(1+\max_{[0,1]}\theta^2\right)\left(\int_0^1v_x^2dx\right)ds\\[2mm]
&\leq
C+C\int_0^t\left(1+\max_{[0,1]}\theta^2\right)\left(\int_0^1v_x^2dx\right)ds.
\end{split}
\end{equation*}
Thus, the lemma is completed by Gr\"{o}nwall's inequality.
\end{proof}
%%%%%%%%%%%%%%%%%%%%%%%%%%%%%%%%%%%%%%%%%%%%%%%%%%%%%%%%%%%%%%%%%%%%%%%%%%%%%
%%%%%%%%%%%%%%%%%%%%%%%%%%%%%%%%%%%%%%%%%%%%%%%%%%%%%%%%%%%%%%%%%%%%%%%%%%%%%%
We introduce the following new variables
\begin{align}
&X:=\int_0^t\int_0^1(1+\theta^{q})\theta_t^2dxds,\label{X}\\
&Y:=\max_{0\leq t\leq
T}\int_0^1(1+\theta^{2q})\theta_x^2dx,\label{Y}\\
&Z:=\max_{0\leq t\leq T}\int_0^1u_{xx}^2dx.\label{Z}
\end{align}

Firstly, with the aid of interpolation and embedding theorem, we
deduce that
\begin{align*}
\left|u\right|^{(0)}\leq C\left(1+Z^{\frac38 }\right)
\end{align*}
by simple calculation.

%%%%%%%%%%%%%%%%%%%%%%%%%%%%%%%%%%%%%%%%%%%%%%%%%%%%%%%%%%%%%%%%%%%%%%%%%%%%%%%%%%%%%%%%%%
%%%%%%%%%%%%%%%%%%%%%%%%%%%%%%%%%%%%%%%%%%%%%%%%%%%%%%%%%%%%%%%%%%%%%%%%%%%%%%%%%%%%%%%%%%55
\begin{lemma}
Under the assumptions of Theorem \ref{thm}, we have
\begin{equation}
\max_{[0,1]\times[0,t]}\theta(x,t)\leq
C+CY^{\frac{1}{2(q+2n-1)}},\label{Lemma82}
\end{equation}
for $(x,t)\in (0,1)\times (0,T)$.
\end{lemma}
%%%%%%%%%%%%%%%%%%%%%%%%%%%%%%%%%%%%%%%%%%%%%%%%%%%%%%%%%%%%%%%%%%%%%%%%%%%%%%%%%%%%%%%%%%%%%%%%%%%%%
\begin{proof}
By the embedding theorem, we deduce that
\begin{equation}
\max_{[0,1]}\theta^{q+2n-1}(x,t)  \leq
C\int_0^1\theta^{q+2n-1}dx+C_{n}\int_0^1(1+\theta)^{q+2n-2}|\theta_x|dx,\notag
\end{equation}
which  implies
 that
\begin{equation}
\begin{split}
&\max_{[0,1]}\theta^{q+2n-1}(x,t)\\
&\leq C\max_{[0,1]}\theta^{q+2n-2}\int_0^1\theta
dx+C_n\int_0^1(1+\theta)^{q+2n-2}\left|\theta_x\right|
dx\\
&\leq \varepsilon \max_{[0,1]}\theta^{q+2n-1}
+C_n\left(\int_0^1(1+\theta)^{2q}\theta_x^2dx\right)^{\frac12}
\left(\int_0^1(1+\theta)^{4n-4}dx\right)^{\frac12}+C_{\varepsilon}.\notag
\end{split}
\end{equation}
So we deduce \eqref{Lemma82}.
\end{proof}
%%%%%%%%%%%%%%%%%%%%%%%%%%%%%%%%%%%%%%%%%%%%%%%%%%%%%%%%%%%%%%%%%%%%%%%%%%%%%%%%%%%%%%%%%%%%%%%%%%%
%%%%%%%%%%%%%%%%%%%%%%%%%%%%%%%%%%%%%%%%%%%%%%%%%%%%%%%%%%%%%%%%%%%%%%%%%%%%%%%%%%%%%%%%%%%%%%%%%%%%%5

%%%%%%%%%%%%%%%%%%%%%%%%%%%%%%%%%%%%%%%%%%%%%%%%%%%%%%%%%%%%%%%%%%%%%%%%%%%%%%%%%%%%%%%%%%%%%%%%%%%%%%%%%%%%%%%%%%
%%%%%%%%%%%%%%%%%%%%%%%%%%%%%%%%%%%%%%%%%%%%%%%%%%%%%%%%%%%%%%%%%%%%%%%%%%%%%%%%%%%%%%%%%%%%%%%%%%%%%%%%%%%%%%%%%%%%
\begin{lemma}
 One has
\begin{equation}
X+Y\leq C\left(1+Z^{\frac12}\right). \label{lemma94}
\end{equation}
for  $0\leq\beta<q+9$ and $q\geq 0$.
\end{lemma}
\begin{proof}
We introduce the function as in \cite{kawohl,umeharaTani}
\begin{equation}
K(v,\theta):=\int_0^{\theta}\frac{\kappa(v,\xi)}{v}d\xi. \notag
\end{equation}
The simple calculation leads to
\begin{align}
&K_t=\frac{\kappa}{v}\theta_t+K_vu_x, \label{K1}\\
&K_{xt}=\left(\frac{\kappa}{v}\theta_x\right)_t+K_{vv}v_{x}u_{x}+K_{v}u_{xx}+\left(\frac{\kappa}{v}\right)_vv_x
\theta_t,\label{K2}\\
&|K_v|, |K_{vv}|\leq C\left(1+\theta\right)^{q+1}.\label{K3}
\end{align}
Multiplying \eqref{sub3} by $K_t$ and integrating it over
$(0,1)\times (0,t)$, we get
\begin{equation}
\begin{split}
&\int_0^t\int_0^1\frac{\kappa e_{\theta}\theta_t^2}{v}dxd\tau
+\int_0^t\int_0^1\frac{\kappa}{v}\theta_x\left(\frac{\kappa}{v}\theta_x\right)_tdxd\tau\\
&=-\int_0^t\int_0^1e_{\theta}\theta_tK_v
u_xdxd\tau-\int_0^t\int_0^1\left(\theta
p_{\theta}u_x-\frac{\mu}{v}u_x^2\right)\frac{\kappa}{v}\theta_tdxd\tau\\
&- \int_0^t\int_0^1\left(\theta
p_{\theta}u_x-\frac{\mu}{v}u_x^2\right)K_vu_xdxd\tau\
-\int_0^t\int_0^1\frac{\kappa}{v}\theta_xK_{vv}v_xu_xdxd\tau\\
&-\int_0^t\int_0^1\frac{\kappa}{v}\theta_x K_v u_{xx}dxd\tau
-\int_0^t\int_0^1\frac{\kappa}{v}\theta_x\left(\frac{\kappa}{v}\right)_v
v_{x}\theta_tdxd\tau\\
&+\lambda\int_0^t\int_0^1 \phi z^mK_t dxd\tau.\label{lemma91}
\end{split}
\end{equation}

In the sequel, we will estimate all terms in \eqref{lemma91}.

Firstly, we have
\begin{equation}
\int_0^t\int_0^1\frac{\kappa e_{\theta}\theta_t^2}{v}dxd\tau \geq C
\int_0^t\int_0^1(1+\theta^3)(1+\theta^q)\theta_t^2dxd\tau\geq CX,
\end{equation}
and
\begin{equation}
\begin{split}
&\int_0^t\int_0^1\frac{\kappa}{v}\theta_x\left(\frac{\kappa}{v}\theta_x\right)_{t}dxd\tau\\
&=\frac12\int_0^1\left(\frac{\kappa}{v}\theta_x\right)^2dx
-\frac12\int_0^1\left(\frac{\kappa}{v}\theta_x\right)^2(x,0)dx\\
&\geq C\int_0^1(1+\theta^{q})^2\theta_x^2dx-C\geq CY-C,
\end{split}
\end{equation}
by the definition of \eqref{X} and \eqref{Y}.

Secondly, we find that
\begin{equation}
\begin{split}
&\left|\int_0^t\int_0^1e_{\theta}\theta_tK_v u_xdxd\tau\right| \\
&\leq C\int_0^t\int_0^1\left(1+\theta\right)^{q+4}|\theta_tu_x|dxd\tau\\
&\leq \varepsilon X
+C_{\varepsilon}Y^{\frac{q+8}{2(q+2n-1)}}\int_0^t\int_0^1u_x^2dxd\tau\\
&\leq \varepsilon (X+Y)+C,\notag
\end{split}
\end{equation}

Similarly, we have
\begin{equation}
\begin{split}
&\left|\int_0^t\int_0^1\left(\theta
p_{\theta}u_x-\frac{\mu}{v}u_x^2\right)\frac{\kappa}{v}\theta_tdxd\tau\right|\\
&\leq C(T)\int_0^t\int_0^1\left[(1+\theta)^{q+4}|u_x\theta_t|+(1+\theta)^q|\theta_t|u_x^2\right]dxd\tau\\
&\leq \varepsilon
X+C_{\varepsilon}\left|\left(1+\theta^{q+8}\right)\right|^{(0)}\int_0^t\int_0^1u_x^2dxd\tau\\
&+C_{\varepsilon}|(1+\theta)^{q}|^{(0)}\int_0^t\int_0^1u_x^4dxd\tau\\
&\leq \varepsilon (X+Y). \notag
\end{split}
\end{equation}
and
\begin{equation}
\begin{split}
&\left|\int_0^t\int_0^1\left(\theta
p_{\theta}u_x-\frac{\mu}{v}u_x^2\right)K_vu_xdxd\tau\right|\\
&\leq C\int_0^t\int_0^1\left[(1+\theta)^{q+5}u_x^2+(1+\theta)^{q+1}
\left|u_x\right|^3\right]dxd\tau\\
&\leq C
\left|(1+\theta^{q+5})\right|^{(0)}\int_0^t\int_0^1u_x^2dxd\tau\\
&+\left|\left(1+\theta\right)^{q+1}\right|^{(0)}\int_0^t\int_0^1|u_x|^3dxd\tau+C\\
&\leq \varepsilon Y+C.\notag
\end{split}
\end{equation}

Likewise, we get
\begin{equation}
\begin{split}
&\left|\int_0^t\int_0^1\frac{\kappa}{v}\theta_xK_{vv}v_xu_xdxd\tau\right|\\
&\leq C\int_0^t\int_0^1\left(1+\theta\right)^{2q+1}|\theta_xv_xu_x|dxd\tau\\
&\leq C\left|u_x\right|^{(0)}
\left(\int_0^t\int_0^1\frac{\kappa\theta_x^2}{v\theta^2}dxd\tau\right)^{\frac12}
\left(\int_0^t\max_{[0,1]}(1+\theta)^{3q+4}\int_0^1v_{x}^2dxd\tau\right)^{\frac12}\\
&\leq
CZ^{\frac38}.\notag
\end{split}
\end{equation}

In the same manner, we find that
\begin{equation}
\begin{split}
&\left|\int_0^t\int_0^1\frac{\kappa}{v}\theta_x K_v u_{xx}dxd\tau\right|\\
&\leq
C\left(\int_0^t\int_0^1\frac{\kappa\theta_x^2}{v\theta^2}dxd\tau\right)^{\frac12}
\left(\int_0^t\max_{[0,1]}(1+\theta)^{3q+4}\int_0^1u_{xx}^2dxd\tau\right)^{\frac12}\\
&\leq CZ^{\frac12}.\\
\label{Lemma103}\notag
\end{split}
\end{equation}

At last, one has
\begin{equation}
\begin{split}
&\left|\int_0^t\int_0^1\frac{\kappa}{v}\theta_x\left(\frac{\kappa}{v}\right)_v
v_{x}\theta_tdxd\tau\right|\\
&\leq C\int_0^t\int_0^1\left(1+\theta\right)^{2q}\left|\theta_x v_x \theta_t\right|dxd\tau\\
&\leq \varepsilon
X+C_{\varepsilon}\int_0^t\max_{[0,1]}\left(\frac{\kappa\theta_x}{v}\right)^2
\left(\int_0^1\frac{(1+\theta)^{3q}v_x^2v^2}{\kappa^2}dx\right)d\tau\\
&\leq \varepsilon
X+C_{\varepsilon}|\left(1+\theta\right)^{q}|^{(0)}\int_0^t\max_{[0,1]}\left(\frac{\kappa\theta_x}{v}\right)^2\left(\int_0^1v_x^2dx\right)d\tau\\
&\leq \varepsilon
X+C_{\varepsilon}Y^{\frac{q}{2(q+2n-1)}}\int_0^t\max_{[0,1]}\left(\frac{\kappa\theta_x}{v}\right)^2d\tau.
\label{lemma92}
\end{split}
\end{equation}

The next goal  is to show  the estimation of the last term on the
right hand side of the above inequality. From the embedding theorem,
it yields
\begin{equation}
\begin{split}
&\int_0^t\max_{[0,1]}\left(\frac{\kappa\theta_x}{v}\right)^2d\tau\\
&\leq
C\int_0^t\left(\int_0^1\left(\frac{\kappa\theta_x}{v}\right)^2dx
+\int_0^1\left|\frac{\kappa\theta_x}{v}\cdot\bigg(\frac{\kappa\theta_x}{v}\bigg)_x\right|dx\right)d\tau,\\
&\leq
C\int_0^t\int_0^1\left(\frac{\kappa\theta_x}{v}\right)^2dxd\tau
+C\left(\int_0^t\int_0^1\left(\frac{\kappa\theta_x}{v}\right)^2dxd\tau\right)^{\frac12}
\left(\int_0^t\int_0^1\left(\frac{\kappa\theta_x}{v}\right)_x^2dxd\tau\right)^{\frac12}\\
&\leq CY^{\frac{q+2}{2(q+2n-1)}}\int_0^t\int_0^1\frac{\kappa
\theta_x^2}{v\theta^2}dxd\tau\\
&+CY^{\frac{q+2}{4(q+2n-1)}}\left(\int_0^t\int_0^1\frac{\kappa
\theta_x^2}{v\theta^2}dxd\tau\right)^{\frac12}\left(\int_0^t\int_0^1\left(\frac{\kappa\theta_x}{v}\right)_x^2dxd\tau\right)^{\frac12}
\\
&\leq
CY^{\frac{q+2}{2(q+2n-1)}}+CY^{\frac{q+2}{4(q+2n-1)}}\left(\int_0^t\int_0^1(e_{\theta}^2\theta_t^2
 +\theta^2
p_{\theta}^2u_x^2+u_x^4+\phi^2z^{2m})dxds\right)^{\frac12}\label{lemma101}
\end{split}
\end{equation}

In the following, we derive the estimations of all terms of the
right hand side in \eqref{lemma101}.

Firstly, from \eqref{X} and \eqref{Lemma82}, we get
\begin{equation*}
\begin{split}
&\int_0^t\int_0^1(1+\theta)^{q+2}e_{\theta}^2\theta_{t}^2dxd\tau\\
&\leq
C\left|(1+\theta)^8\right|^{(0)}\int_0^t\int_0^1\left(1+\theta\right)^{q}\theta_{t}^2dxd\tau\\[2mm]
&\leq C\left|(1+\theta)^8\right|^{(0)}X\\[2mm]
&\leq C \left(X+Y^{\frac{8}{2(q+2n-1)}}X\right).\label{Lemma101}
\end{split}
\end{equation*}
Similarly, recalling \eqref{Lemma71} and \eqref{Lemma82}, we deduce
that
\begin{equation*}
\begin{split}
&\int_0^t\int_0^1(1+\theta)^{q+2}\theta^2 p_{\theta}^2u_x^2dxd\tau\\
&\leq C\int_0^t\int_0^1\left(1+\theta\right)^{q+10}u_x^2dxd\tau\\
&\leq C\left(1+Y^{\frac{q+10}{2(q+2n-1)}}\right),
\end{split}
\end{equation*}
and
\begin{equation*}
\begin{split}
&\int_0^t\int_0^1\left(1+\theta\right)^{q+2}u_x^4dxd\tau\\
&\leq C\left|(1+\theta)^{q+2}\right|^{(0)}\int_0^t\int_0^1u_x^4dxd\tau\\
&\leq C\left(1+Y^{\frac{q+2}{2(q+2n-1)}}\right).
\end{split}
\end{equation*}
At the same time, we obtain
\begin{equation*}
\begin{split}
&\int_0^t\int_0^1\left(1+\theta\right)^{q+2}\phi^2z^{2m}dxd\tau \\
&\leq
C\left|(1+\theta)^{q+2+2\beta}\right|^{(0)}\\
&\leq CY^{\frac{q+2+2\beta}{2(q+2n-1)}}+C, \label{Lemma102}
\end{split}
\end{equation*}
Thus
\begin{equation}
\begin{split}
&Y^{\frac{q}{2(q+2n-1)}}\int_0^t\max_{[0,1]}\left(\frac{k}{v}\theta_x\right)^2d\tau\\
 &\leq CY^{\frac{q+1}{q+2n-1}}+
CY^{\frac{3q+2}{4(q+2n-1)}}\left(1+Y^{\frac{8}{4(q+2n-1)}}X^{\frac12}+X^{\frac12}
+Y^{\frac{q+10}{4(q+2n-1)}}+Y^{\frac{q+2+2\beta}{4(q+2n-1)}}\right)\\
&\leq \varepsilon (X+Y)+C, \notag
\end{split}
\end{equation}
which together with \eqref{lemma92} leads to
\begin{equation}
\left|\int_0^t\int_0^1\frac{k}{v}\theta_x\left(\frac{k}{v}\right)_v
v_{x}\theta_tdxd\tau\right|\leq \varepsilon (X+Y)+C.
\label{Lemma104}
\end{equation}
Finally, the term $\lambda\int_0^t\int_0^1\phi z^m K_tdxd\tau$
gives,
\begin{equation}
\begin{split}
&\lambda\int_0^t\int_0^1 \phi z^m K_tdxd\tau\\
 &\leq C\int_0^t\int_0^1\phi z^mk\left|\theta_t\right|dxd\tau
+C\int_0^t\int_0^1\theta^{q+1}\left| u_x\right|\phi z^mdxd\tau\\
&\leq \varepsilon
X+\left|(1+\theta)^{q+\beta}\right|^{(0)}\int_0^t\int_0^1\phi z^{2m}dxd\tau
+C\left|\theta^{q+\beta+1}\right|^{(0)}\int_0^t\int_0^1\left|u_x\right|dxd\tau\\
&\leq \varepsilon X+CY^{\frac{q+\beta+1}{2(q+2n-1)}}\\
 &\leq \varepsilon \left(X+Y\right)+C. \label{lemma93}
\end{split}
\end{equation}
Collecting \eqref{lemma91}--\eqref{lemma93}, we deduce
\eqref{lemma94}.
\end{proof}
%%%%%%%%%%%%%%%%%%%%%%%%%%%%%%%%%%%%%%%%%%%%%%%%%%%%%%%%%%%%%%%%%%%%%%%%%%%%%%%%%%%%%%%%%%%%%%%%%%%%%%%%%%%%%%%
%%%%%%%%%%%%%%%%%%%%%%%%%%%%%%%%%%%%%%%%%%%%%%%%%%%%%%%%%%%%%%%%%%%%%%%%%%%%%%%%%%%%%%%%%%%%%%%%%%%%%%%%%%%%%%%
\begin{lemma}\label{lemma1313}
We have
\begin{equation*}
\int_0^1\left(u_x^2+\theta_x^2+u_{xx}^2+
u_t^2\right)dx+\int_0^t\int_0^1\left(\theta_t^2+u_{xt}^2\right)dxd\tau\leq
C(T),
\end{equation*}
and
\begin{equation*}
\left|u_x\right|^{(0)}+\left|u\right|^{(0)}+\left|\theta\right|^{(0)}\leq
C(T). \label{lemma3111}
\end{equation*}
\end{lemma}
\begin{proof}
We differentiate \eqref{sub2} by $t$ and then multiply the equality
by $u_t$ to obtain  by integrating it over $[0,1]$,
\begin{equation}
\frac{d}{dt}\int_0^1\frac{u_t^2}{2}dx+\mu
\int_0^1\frac{u_{xt}^2}{v}dx=\int_0^1\left(p_tu_{xt}+\frac{\mu}{v^2}u_x^2u_{xt}\right)dx.\notag
\end{equation}
 Integrating it over $[0,t]$ and using the Young inequality, we
get
\begin{equation}
\begin{split}
&\int_0^1u_t^2dx+\int_0^t\int_0^1u_{xt}^2dxds\\[2mm]
&\leq
C+C\int_0^t\int_0^1\left(p_t^2+u_x^4\right)dxd\tau\\[2mm]
&\leq C+C\int_0^t\int_0^1(1+\theta^6)\theta_t^2dxd\tau
+C\left|\theta^2\right|^{(0)}\int_0^t\int_0^1u_x^2dxd\tau\\[2mm]
&~~+C\int_0^t\int_0^1u_{x}^4dxd\tau. \notag
\end{split}
\end{equation}
 and hence
\begin{equation}
\begin{split}
&\int_0^1u_t^2dx+\int_0^t\int_0^1u_{xt}^2dxds\\[2mm]
&\leq C\left(1+X+Y^{\frac{3}{q+2n-1}}X+Y^{\frac{1}{q+2n-1}}\right)\\[2mm]
&\leq C\left(1+Z^{\frac12\times \frac{q+2n+2}{q+2n-1}}\right).\notag
\end{split}
\end{equation}
On the other hand, the momentum equation \eqref{sub2} implies
\begin{equation}
\begin{split}
\int_0^1u_{xx}^2dx
&\leq C+C\int_0^1(u_t^2+p_x^2+u_x^2v_x^2)dx\\
 &\leq C
\left(1+\int_0^1u_t^2dx+\int_0^1\left(1+\theta^6\right)\theta_x^2dx
+\left(\left|\theta^2\right|^{(0)}+\left|u_x^2\right|^{(0)}\right)\int_0^1v_x^2dx\right)\\
&\leq
C\left(1+Z^{\frac12\times \frac{q+2n+2}{q+2n-1}}+Z^{\frac34}\right)\\
&\leq
C\left(1+Z^{\frac34}\right)\\
 \notag
\end{split}
\end{equation}
It follows from this that
\begin{equation}
Z\leq C\left(1+Z^{\frac34}\right),\notag
\end{equation}
which implies the boundedness of $Z$ by Young inequality  when
$q\geq 0$ and $n$ is sufficiently large. So are $X$ and $Y$  by
\eqref{lemma94}. Furthermore, it holds for $|u_x|^{(0)}$,
$|\theta|^{(0)}$, $\int_0^1(\theta_x^2+u_t^2+u_{xx}^2)dx$ and
$\int_0^t\int_0^1(\theta_t^2+u_{xt}^2)dxd\tau$. so is $u$ by the
embedding theorem. This is complete.
\end{proof}
%%%%%%%%%%%%%%%%%%%%%%%%%%%%%%%%%%%%%%%%%%%%%%%%%%%%%%%%%%%%%%%%%%%%%%%%%%%%%%%%%%%%%%%%%%%%%%%%%%%%%%%%%%%%%%%%%%%%%%%%%%%%%%%%%
%%%%%%%%%%%%%%%%%%%%%%%%%%%%%%%%%%%%%%%%%%%%%%%%%%%%%%%%%%%%%%%%%%%%%%%%%%%%%%%%%%%%%%%%%%%%%%%%%%%%%%%%%%%%%%%%%%%%%%%%%%%%%%%%

%%%%%%%%%%%%%%%%%%%%%%%%%%%%%%%%%%%%%%%%%%%%%%%%%%%%%%%%%%%%%%%%%%%%%%%%%%%%%%%%%%%%%%%%%%%%%%
%%%%%%%%%%%%%%%%%%%%%%%%%%%%%%%%%%%%%%%%%%%%%%%%%%%%%%%%%%%%%%%%%%%%%%%%%%%%%%%%%%%%%%%%%
With the aid of the above a priori estimates, we further deduce a
priori estimates of  higher derivatives of $(\rho,u,\theta,z)$ which
is analogous to the reference \cite{umeharaTani}, and then we omit
it proofs for brevity.
\begin{lemma}
It satisfies when $0\leq \beta<q+9$ and $q\geq 0$.
\begin{align*}
&\theta(x,t)\geq C,\\
&\int_0^1\left(z_x^2+z_{xx}^2+z_t^2\right)dx+\int_0^t\int_0^1z_{xt}^2dxd\tau\leq
C,\\
&\int_0^1(\theta_{xx}^2+\theta_{xt}^2)dx+\int_0^t\int_0^1\theta_{xt}^2dxd\tau\leq
C,
\end{align*}
\end{lemma}
For the H\"{o}lder estimation of global solutions, we refer to
\cite{umeharaTani} for detailed processes. Thus, the proof of
Theorem \ref{thm} is complete.

As mentioned in Remark \ref{rem22}, the same conclusion are also
satisfied for the external force $p_e\leq 0$ on the free boundaries
provided  that we obtain a priori estimates of specific volume, see
the Appendix. With Lemma \ref{appendix} in hand,  we can get the
same a priori estimates of $(v,u,\theta,z)$.

\section*{Appendix}

\begin{lemma}\label{appendix}
If $p_e\leq  0$, then it has
\begin{equation}
\int_0^1v(x,t)dx\leq C(T),\notag
\end{equation}
 for $0<t\leq T$.
\end{lemma}
\begin{proof}
It is derived from integrating \eqref{sub2} over $[0,x]$,
\begin{equation}
\int_0^xu_tdy=-p+\frac{\mu
u_x}{v}-G\bigg(\frac{x(x-1)}{2}\bigg)+p_e.\nonumber
\end{equation}
which implies by multiplying  $v$ and integrating with respect to
$x$  and $t$ that
\begin{equation}
\begin{split}
&\mu\int_0^1vdx+\frac{G}{2}\int_0^t\int_0^1x(1-x)vdxds\\
&=\mu\int_0^1v_0dx-p_e\int_0^t\int_0^1vdxds\\
&+\int_0^t\int_0^1vpdxds+\int_0^t\int_0^1v\bigg(\int_0^xu_tdy\bigg)dxds.\label{lemma21}
\end{split}
\end{equation}
 Obviously, it yields

\begin{equation}
\begin{split}
&\int_0^t\int_0^1vpdxds\\
&=\int_0^t\int_0^1\bigg(R\theta+\frac{a}{3}v\theta^4\bigg)dxds\\
&\leq C\int_0^t\int_0^1edxds\\
&\leq C(T)-C p_e\int_0^t\int_0^1vdxds. \label{lemma22}
\end{split}
\end{equation}
 On the other hand, we get from \eqref{sub1} by integration
\begin{equation*}
v(x,t)=v_0(x)+\int_0^tu_x(x,s)ds,
\end{equation*}
which leads to
\begin{equation}
\begin{split}
&\int_0^t\int_0^1v\bigg(\int_0^xu_tdy\bigg)dxds\\
&=\int_0^1\bigg(v_0(x)+\int_0^tu_x(x,s)ds\bigg)\bigg(\int_0^xudy\bigg)dx\\
&+\int_0^t\int_0^1u^2dxds-\int_0^1v_0(x)\bigg(\int_0^xu_0(y)dy\bigg)dx\\
&=\int_0^1v_0(x)\bigg(\int_0^x(u-u_0(y))dy\bigg)dx+\int_0^t\int_0^1u^2dxds\\
&-\int_0^1u(x,t)\bigg(\int_0^tu(x,s)ds\bigg)dx,\notag
\end{split}
\end{equation}
 by integrating by parts. Using Young inequality and initial conditions, we
get from \eqref{lemma111}
\begin{equation}
\begin{split}
&\int_0^t\int_0^1v\bigg(\int_0^xu_tdy\bigg)dxds\\
&\leq \varepsilon
\int_0^1u^2dx+C_{\varepsilon}\int_0^t\int_0^1u^2dxds+C_{\varepsilon}\\
&\leq -\varepsilon
p_e\int_0^1vdx-C_{\varepsilon}p_e\int_0^t\int_0^1vdxds+C_{\varepsilon}(T).\label{lemma23}
\end{split}
\end{equation}
Collecting \eqref{lemma21}--\eqref{lemma23}, we have
\begin{equation}
\begin{split}
&(\mu+\varepsilon
p_e)\int_0^1vdx+\frac{G}{2}\int_0^t\int_0^1x(1-x)vdxds\\
&\leq
-C_{\varepsilon}p_e\int_0^t\int_0^1vdxds+C_{\varepsilon}(T).\notag
\end{split}
\end{equation}
 Choosing sufficiently small $\varepsilon$ and by Gronwall
inequality, one obtains the desired result for $p_e\leq0$. The proof
is complete.
\end{proof}
%%%%%%%%%%%%%%%%%%%%%%%%%%%%%%%%%%%%%%%%%%%%%%%%%%%%%%%%%%%%%%%%%%%%%%%%%%%%%%%%%%%%%%%%%%%%%%%%%%%%%%%
\section*{Acknowledgements} This work is partially supported by
National Nature Science Foundation of China (NNSFC-10971234,
NNSFC-10962137,NNSFC-11001278), China Postdoctoral Science
Foundation funded project (NO. 20090450191), Natural Science
Foundation of GuangDong (NO. 9451027501002564) and the Fundamental
Research Funds for the Central Universities.

%%%%%%%%%%%%%%%%%%%%%%%%%%%%%%%%%%%%%%%%%%%%%%%%%%%%%%%%%%%%%%%%%%%%%%%%%%%%%%%%%%%%%%%%%%%%%%%%%%%%%


\begin{thebibliography}{99}
\bibitem{Bressan}
A. Bressan, Global solutions for the one dimensional equations of a
viscous reactive gas, Bollettino U.M.I., \textbf{5} (1986) 291--308.

\bibitem{BebernesBressan82}
J. Bebernes and A. Bressan, Thermal behavior for a confined reactive
gas, J. Differential Equations, \textbf{44} (1982) 118--133.

\bibitem{BebernesBressan85}
J. Bebernes and A. Bressan, Global a priori estimates for a viscous
reactive gas, Proc. Roy. Soc. Edinb., \textbf{101A} (1985) 321--333.


 \bibitem{Chen}
 G.Q. Chen, Global solution to the compressible Navier-Stokes
 equations for a reacting mixture, SIAM J. Math. Anal., \textbf{23} (1992)
 609--634.

\bibitem{ChenHoffTrivisaCPDE}
 G.Q. Chen, D. Hoff and K. Trivisa,
 Global solutions of the compressible Navier-Stokes equations with large
discontinuous initial data, Comm. Partial Differential Equations,
\textbf{25} (2000) 2233--2257.

\bibitem{ChenHoffTrivisa1}
 G.Q. Chen, D. Hoff and K. Trivisa,
 On the Navier-Stokes equations for exothermically reacting
 compressible fluids, Acta Math. Appl. Sinica., \textbf{18} (2002) 15--36.

 \bibitem{ChenHoffTrivisa2}
 G.Q. Chen, D. Hoff and K. Trivisa,
 Global solution to a model for exothermically reacting compressible
 flows with large discontinuous data, Arch. Ration. Mech.
Anal., \textbf{166} (2003) 321--358.

\bibitem{ChenWagner}
G.Q. Chen and D. H. Wagner, Global entropy solutions to
exothermically reacting, compressible Euler equations, J.
Differential Equations, \textbf{191} (2003) 277--322.

\bibitem{DafermosHisao}
C.M. Dafermos and L. Hisao, Global smooth thermomechanical processes
in one-dimensional thermoviscoelasticity,  Nolinear Anal.,
\textbf{6} (1982) 435--454.

\bibitem{Donatelli06}
D. Donatelli and K. Trivisa, On the motion of a viscous compressible
radiative reacting gas,  Comm. Math. Phys., \textbf{265} (2006)
463--491.

\bibitem{Donatelli07}
D. Donatelli and K. Trivisa, A multidimensional model for the
combustion of compressible fluids,  Arch. Ration. Mech. Anal.,
\textbf{185} (2007)  379--408.

\bibitem{Ducomet3M99}
B. Ducomet, Some asymptotics for a reactive Navier-Stokes-Poisson
system, Math. Models Methods Appl. Sci., \textbf{9} (1999)
1039--1076.

\bibitem{Ducomet99}
B. Ducomet, A model of thermal dissipation for a one-dimensional
viscous reactive and radiative gas, Math. Mech. Appl.
Sci.,\textbf{22} (1999)  1323--1349.

\bibitem{Ducomet02}
B. Ducomet, Global existence for a nuclear fluid in one dimension:
the $T>0$ case, Appl. Math., \textbf{47} (2002) 45--75.





\bibitem{DucometZlotnikhigher}
B. Ducomet and A. Zlotnik, On the large-time behavior of 1D
radiative and reactive viscous flows for higher-order kinetics,
Nonlinear Anal.,  \textbf{63} (2005) 1011--1033.

\bibitem{DucometZlotnik}
B. Ducomet and A. Zlotnik, Lyapunov functional  method for 1D
radiative and reactive viscous gas dynamics, Arch. Ration. Mech.
Anal., \textbf{177} (2005) 185--229.

\bibitem{Guo}
B. Guo and P. Zhu, Asymptotic behavior of the solution to the system
for a viscous reactive gas, J. Differential Equations, \textbf{155}
(1999) 177--202.


\bibitem{kawohl}
B. Kawohl, Global existence of large solutions to initial boundary
value problems for a viscous, heat-conducting, one-dimensional real
gas, J. Differential Equations, \textbf{58} (1985) 76--103.



\bibitem{Jiang}
S. Jiang, On initial boundary value problem for a viscous,
heat-conducting, one-dimensional real gas, J. Differential
Equations, \textbf{110} (1994) 157--181.

\bibitem{qin}
Yuming Qin, Guli Hu and Taige Wang, Global smooth solutions for the
compressible viscous and heat-conductive gas, to appear in Quarterly
Appied Mathematics.



\bibitem{Secchi}
P. Secchi, On the motion of gaseous stars in the presence of
radiation, Comm. Partial Differential Equations, \textbf{15} (1990)
185--204.

\bibitem{umeharaTani}
M. Umehara and A. Tani, Global solution to the one-dimensional
equations for a self-gravitating viscous radiative and reactive gas,
J. Differential Equations, \textbf{234} (2007) 439--463.

\bibitem{umeharaTani11}
M. Umehara and A. Tani,  Temporally global solutions to the
equations for a spherically symmetric viscous radiative and reactive
gas over the rigid core, Analysis and Applications, \textbf{6}
(2008) 1--29.


\bibitem{Yanagi98}
S. Yanagi, Asymptotic stability of the solutions to a full
one-dimensional system of heat-conducting reactive compressible
viscous gas, Japan J. Indust. Appl. Math., \textbf{15} (1998)
423--442.

\bibitem{zhangjianwen}
Wen Zhang and Jianwen Zhang, On the growth condition of
heat-conductiviy in 1-D radiative and reactive viscous gas dynamics,
preprint.

\bibitem{ZR}
Ya.B. Zel'dovich and Yu.P. Raiser, Physics of shock waves and high-
temperature hydrodynamic phenomena. Dover Publications, Inc.,
Mineola,

\end{thebibliography}
\end{document}